\documentclass[11pt]{amsart}

\usepackage{amsmath, amsthm, amssymb, amsfonts, amscd, latexsym, wasysym}

\usepackage[T1]{fontenc}            
\usepackage{palatino}               
\usepackage{comment}

\usepackage[abs]{overpic}		  

\usepackage{xcolor}
\definecolor{indigo}{rgb}{0.29, 0.0, 0.51}  
\usepackage[colorlinks, urlcolor=indigo, linkcolor=indigo, citecolor=indigo]{hyperref}

\usepackage[hcentering, top = 1.5in, total={6in, 8.3in}]{geometry}  

\theoremstyle{plain}
\newtheorem{theorem}{Theorem}

\newtheorem{lemma}[theorem]{Lemma}

\theoremstyle{definition}

\theoremstyle{remark}
\newtheorem{remark}[theorem]{Remark}

\newtheorem{problem}[theorem]{Problem}
\newtheorem{fact}[theorem]{Fact}

\numberwithin{theorem}{section}

\newcommand{\Z}{\mathbb{Z}}           
\newcommand{\CP}{\mathbb{CP}}         

\makeatletter
\newcommand*\bigcdot{\mathpalette\bigcdot@{0.6}}
\newcommand*\bigcdot@[2]{\mathbin{\vcenter{\hbox{\scalebox{#2}{$\m@th#1\bullet$}}}}}
\makeatother

\usepackage{tikz, pgfplots} 
    \pgfplotsset{compat=1.18} 
    \usetikzlibrary{arrows,shapes}
    \usetikzlibrary{fit,shapes.geometric} 
    \usetikzlibrary{decorations.pathreplacing,angles,quotes} 

\begin{document}

\title{A note on the topology of Lefschetz fibrations}

\author{Sierra Knavel}


\address{School of Mathematics \\ Georgia Institute of Technology \\ Atlanta, GA}
\email{sknavel3@gatech.edu} 



\begin{abstract}
We prove an upper bound for the first Betti number of a nontrivial genus-$g$ Lefschetz fibration. We also show that if the monodromy of a Lefschetz fibration is transitive with respect to the mapping class group, the Lefschetz fibration is simply connected. Lastly, we discuss a potential family of indecomposable genus-2 Lefschetz fibrations with maximally non-trivial first homology which would be candidates for large fundamental group computations, if they exist.
\end{abstract}

\maketitle

\section{Introduction}

The dual results of Donaldson \cite{Donaldson1998} and Gompf \cite{Gompf_01} show that Lefschetz pencils, which can be blown up to Lefschetz fibrations over the 2-sphere, are closely related to closed symplectic 4-manifolds. Given this connection, there has been a great deal of study of symplectic manifolds using Lefschetz fibrations. We are interested in studying the topological properties of a manifold that can be extracted from a Lefschetz fibration. We begin by considering the bounds on the first Betti number. 
Suppose $f: X \rightarrow S^2$ is a nontrivial genus-$g$ Lefschetz fibration structure on $X$, a smooth, closed, compact 4-manifold. 

\begin{theorem}\label{ThmBettinum}
    Let $X$ be a smooth, closed, compact 4-manifold and $f: X \rightarrow S^2$ be a nontrivial genus-$g$ Lefschetz fibration, then 
    \begin{equation}
        0 \leq b_1(X) \leq 2g-2.    
    \end{equation}
\end{theorem}

The bound for Theorem~\ref{ThmBettinum} is sharp for $g=2$ as there are Lefschetz fibrations which realize $b_1 = 0,1, \text{and } 2$ found in \cite{Baykur_Korkmaz_2017, Amoros_ABKP_99, Korkmaz2009, Ozbagci_Stipsicz_2000}. This result improves on the bound previously uncovered by Smith \cite{Smith_Hodge_1999} and Stipsicz \cite{Stipsicz_Chern_2000} who discover that in a nontrivial Lefschetz fibration, at least one vanishing cycle must be nonseparating, which implies that $0 \leq b_1(X) \leq 2g-1$. Stipsicz improves this result to say that there are at least $\frac{4g+2}{5}$ vanishing cycles in a genus-$g$ Lefschetz fibration \cite{Stipsicz_99}. However, they do not comment on the possible homological behavior of the curves. We will see in the proof of Theorem~\ref{ThmBettinum} that at least two of these $\frac{4g+2}{5}$ vanishing cycles must be nonhomologous and nonseparating.

This Theorem is generalized to a Lefschetz fibration with a higher genus base space in Theorem~\ref{Thm:GeneralizedBettinum}. The new inequality becomes the following.%
\[ 0 \leq b_1(X) \leq 2g + 2h - 2 \]
Here, $h$ is the genus of the base space. 

Theorem~\ref{ThmBettinum} is motivated by the conjecture that all nontrivial genus-2 Lefschetz fibrations have Abelian fundamental group with at most two generators. Now we can conclude that in addition to the fundamental group having at most three generators, its Abelianization must have at most rank two. It is also an improvement toward \cite[Question~20]{Baykur_2022} which conjectures that the upper bound on $b_1(X)$ can be improved to $g+1$ for a nontrivial genus-$g$ Lefschetz fibration. The author will explore sharpening this bound for $g>2$.

Gompf shows that every finitely presented group is the fundamental group of some closed symplectic 4-manifold \cite{Gompf_95}. Amorós, Bogomolov, Katzarkov, Pantev, and Smith give an explicit method for constructing a symplectic Lefschetz fibration whose total space has fundamental group isomorphic to a given finitely presented group \cite{Amoros_ABKP_99}. Korkmaz gives another construction of this result using Matsumoto's family of genus-$g$ Lefschetz fibrations \cite{Korkmaz2009}. In Korkmaz's construction, the vanishing cycles for a fixed genus are known and computations of the fundamental group are shown directly. It is conjectured that for genus-2 Lefschetz fibrations, it is not possible for the fundamental group to have more than two generators.  One instance of this conjecture appears in Korkmaz's paper \cite{Korkmaz2009}. Resolving this conjecture is difficult without a complete understanding of which genus-2 Lefschetz fibrations exist, which are indecomposable, and what restrictions could therefore exist on their fundamental groups. It is also conjectured that the fundamental groups of genus-2 Lefschetz fibrations are Abelian \cite{Korkmaz_Stipsicz_2009}, as these are the only kinds of fundamental groups that have appeared in computations. To work toward these conjectures, consider the following question.

\begin{problem}{\label{Prob:hombasis}}
    For nontrivial genus-2 Lefschetz fibrations with $b_1(X)=2$, are there two homology classes of vanishing cycles $[\eta_1]$ and $[\eta_2 ]$ such that $[\eta_1]$ and $[\eta_2 ]$ are a part of an integral basis for homology?  
\end{problem}

Resolving Problem~\ref{Prob:hombasis} in the positive would signify that the Abelianization of the fundamental group has at most two generators. There are examples where one can find two homologically independent vanishing cycles that do not form part of a basis for $\Z^4$ but can find a different pair where they do. Of these examples are two smallest Lefschetz fibrations of genus-2 with seven \cite{Baykur_Korkmaz_2017} and eight \cite{Matsumoto_2004} vanishing cycles, both with fundamental group $\Z^2$. It is currently unknown if there exists indecomposable genus-$g$ Lefschetz fibrations with the same number (and type) of nonseparating and separating vanishing cycles which have nonhomeomorphic total spaces. However, suitably small genus $g=2$ Lefschetz fibrations are rational or ruled surfaces and therefore have diffeomorphic total spaces, see Lemma 8 of \cite{Baykur_Korkmaz_2017} and Equation~\eqref{eq:ns_bounds_BK}. 

We now turn to the relationship between the fundamental group and monodromies of Lefschetz fibrations. Recall that the monodromy of a nontrivial genus-$g$ Lefschetz fibration is \textit{transitive} if it generates the mapping class group of the genus-$g$ fiber.

\begin{theorem}\label{ThmTransitivity}
    Suppose $f : X \rightarrow \Sigma$ where $\Sigma = S^2$ or $D^2$ is a nontrivial genus-$g$ Lefschetz fibration with monodromy $\phi$. If $\phi$ is transitive with respect to the mapping class group of a genus-$g$ surface, then $X$ is simply connected.    
\end{theorem}

\begin{remark}
    If in the theorem, the nontrivial Lefschetz fibration has base space a higher genus surface with any number of boundary components, then the fundamental group of the total space must be the fundamental group of the base space. 
\end{remark}

The condition of a monodromy being transitive has appeared naturally in several places. We now discuss its appearance in results surrounding genus-$2$ Lefschetz fibrations, a fiber genus where the vanishing cycles have the additional structure of being hyperelliptic granted by the hyperellipticity of $Mod(\Sigma_2)$. Siebert and Tian show that a genus-2 Lefschetz fibration is a two-fold branched cover of an $S^2$-bundle over $S^2$ and the transitivity of the monodromy is sufficient and necessary for the connectivity of the branch locus of that covering  \cite{Seibert_Tian_1999}. Siebert and Tian later prove that a genus-2 Lefschetz fibration with only nonseparating vanishing cycles and with transitive monodromy has a holomorphic structure \cite{Siebert_Tian_2005}. Smith reproves a result from the thesis of Chakiris \cite{Chakiris_83} that completely classifies genus-2 Lefschetz fibrations with only nonseparating vanishing cycles and K\"{a}hler total space \cite{Smith_Hodge_1999}. It follows from this result that these objects are simply connected, as their vanishing cycles contain an obvious basis for first homology (see \cite{Farb_Margalit} for a description of the genus-2 chain relations and hyperelliptic involution). Combining the result of Chakiris \cite{Chakiris_83} and Smith \cite{Smith_Hodge_1999} with that of Siebert and Tian \cite{Siebert_Tian_2005} gives the following Lemma for genus-2 Lefschetz fibrations.

\begin{lemma}\label{ThmHolomorphic}
    In a nontrivial genus-2 Lefschetz fibration $f: X \rightarrow S^2$ with monodromy $\phi$ and only nonseparating vanishing cycles, $X$ is holomorphic if and only if $\phi$ is transitive. 
\end{lemma}

\begin{remark}
    A symplectic 4-manifold which has a holomorphic structure is also K\"{a}hler.
\end{remark}

For fiber genus-2, a transitive monodromy is a strong condition. For any genus-2 Lefschetz fibration, Theorem~\ref{ThmTransitivity} and Lemma~\ref{ThmHolomorphic} together say that a holomorphic structure implies simply connectedness.
 
Birman and Hilden \cite{Birman_Hilden_result} show that every element of the mapping class group of the genus-2 surface has a representative that commutes with the hyperelliptic involution. This extra structure of hyperellipticity of the vanishing cycles on a genus-$g$ Lefschetz fibration, which does not hold for all simple closed curves which are vanishing cycles for fibers with genus $g>2$, allows for straightforward computations of invariants of the closed 4-manifold such as the Euler characteristic and signature. Section~\ref{sec:genus_2} discusses the proof of Theorem~\ref{ThmBound_ns} which sharply relates the number of nonseparating and separating vanishing cycles of a genus-2 Lefschetz fibration. This is an improvement on the inequality uncovered by Baykur and Korkmaz \cite{Baykur_Korkmaz_2017} whose equality gives rise to an interesting family of highly non-simply connected Lefschetz fibrations. A (hyperelliptic) genus-2 Lefschetz fibration with $n$ nonseparating vanishing cycles and $s$ separating vanishing cycles is said to be of type $(n,s)$. 

\begin{theorem}\label{ThmBound_ns}
    Suppose $f:X\rightarrow S^2$ is a genus-2 Lefschetz fibration with $n$ nonseparating and $s$ separating vanishing cycles. Then $2n-s \leq 5$ and the family of Lefschetz fibrations of type $(n,s)=(2k, 4k-5), k\geq 2$ is indecomposable, has $b_1 = 2$ for all $k$, and has $b_2 = n+s-2$.  
\end{theorem}

A Lefschetz fibration is \textit{indecomposable} if it cannot be decomposed as the fiber sum of two nontrivial Lefschetz fibrations, further details in Section~\ref{Sec:intro_indecomp}. The only Lefschetz fibration in this family that has been shown to exist is when $k=2,$ a type $(4,3)$. This was found by manipulating the monodromy of the genus-$2$ length forty chain relation \cite{Baykur_Korkmaz_2017} and can be generalized to create higher genera Lefschetz fibrations with small numbers of singular fibers. We remark that Baykur and Akhmedov-Park \cite[Proposition~3]{Akhmedov_Monden_21} have shown that the Lefschetz fibration of type $(6,7)$ is indecomposable (and minimal).

\subsection*{Acknowledgments}
I would like to express appreciation to my advisor, John Etnyre, for his unrelenting support, insight, and criticism as well as his profound belief in my abilities. This work was partially supported by NSF grants DMS-1745583, DMS-2244427, and DMS-2203312. Lastly, I would like to thank Inan\c{c} Baykur and Riccardo Pedrotti for helpful discussions and encouraging the author to consider Lefschetz fibrations with higher genus base.  

\section{Background}
%
\subsection{Lefschetz fibrations}
Solomon Lefschetz first defined a Lefschetz pencil in the 1924 book ``\textit{L’Analysis Situs et la G\'{e}om\'{e}trie Alg\'{e}brique}''. There it was described as a one-parameter family of hyperplane sections of a smooth projective variety which allowed isolated nondegenerate singularities. This structure helped to analyze the topology of complex varieties through vanishing cycles and monodromies.


Results on closed symplectic 4-manifolds often arise from the more accessible Lefschetz fibration structure. Let $\Sigma_{g}$ be a surface of genus-$g$. A \textit{Lefschetz fibration} is a smooth surjection $f:X\rightarrow S^2$ from $X$ a closed, compact 4-manifold. The preimage of a regular value of $f$ is a genus-$g$ surface, $\Sigma_g$, called a \textit{regular fiber}. There are finitely many critical values of $f$, whose preimages are \textit{singular fibers}. A singular fiber contains a critical point that is locally modeled in a complex coordinate chart by the function $f(z,w)= zw$. From this model, one can see that a singular fiber is  homeomorphic to a regular fiber with an embedded $S^1$ collapsed to a point. On nearby regular fibers, this collapsed $S^1$ is an embedded $S^1$ and is called a \textit{vanishing cycle}. 

A directed embedded loop in the base $S^2$ which travels counterclockwise about a critical value $q_i$ will have preimage $f^{-1}(q_i) = [0,1]\times \Sigma_g$ with $\Sigma \times \{0\}$ and $\Sigma \times \{1\}$ identified according to some self-diffeomorphism, $\phi$. This gluing map, called a \textit{monodromy} and denoted $\phi_i$, is supported entirely in an annular region containing the vanishing cycle, called $\eta$, corresponding to $q_i$. This self-diffeomorphism is a right-handed (\textit{positive}) Dehn twist about the vanishing cycle of the corresponding singular fiber, denoted $\tau_{\eta}$.

A Lefschetz fibration is \textit{relatively minimal} if $f: X \rightarrow S^2$ has no fiber containing a 2-sphere with self intersection -1. It is \textit{minimal} if $X$ is smoothly minimal. A 4-manifolds is smoothly minimal if there are no embedded 2-spheres $S \in X$ such that $S \cdot S = -1$.

Lifting a directed embedded loop in the base 2-sphere which encircles several critical values $\{q_1, q_2, ..., q_k \}$ in that order will have monodromy represented by the composition of positive Dehn twists along their corresponding vanishing cycles, $\tau_{\eta_{k}} ... \tau_{\eta_{2}} \tau_{\eta_{1}}$, read right to left. If a vanishing cycle separates a closed genus-$g$ surface into two non-zero genus pieces, the singular fiber is said to be \textit{separating}, or \textit{reducible}. Otherwise, the vanishing cycle is {nonseparating} and the singular fiber is \textit{nonseparating}, or \textit{irreducible}. The lift of an embedded loop which encircles a 2-disk without critical values in $S^2$ will be a trivial $\Sigma_g$-bundle over $S^1$, and therefore equivalent to the identity element in the mapping class group of the fiber, $Mod(\Sigma_g)$. The mapping class group, $Mod(\Sigma_g)$ is the set of self-diffeomorphisms of the genus-$g$ surface $\Sigma_g$ up to isotopy. By abuse of notation, we will denote by $\tau_\gamma$ the Dehn twist and its mapping class, which is the equivalence class of maps isotopic to it. By taking an embedded curve that separates $S^2$ into two disks, one containing all the critical values and one containing no critical values, we see that the monodromy of any Lefschetz fibration over the sphere must be some positive factorization of identity element in $Mod(\Sigma_g)$. 

More generally, if $\{q_1, \ldots, q_k \}$ are the critical values of a Lefschetz fibration $f: X \rightarrow S^2$, then the \textit{global monodromy} of $f$ is the map
\begin{equation*}
    m_f: \pi_1(S^2-\{q_1, \ldots, q_k \}, x_0) \rightarrow Mod(\Sigma_g)
\end{equation*}
where $x_0 \in S^2-\{q_1, \ldots, q_k \}$ and we have identified $f^{-1}(x_0)$ with $\Sigma_g$. The map takes a loop $\gamma$ representing an element of the fundamental group and returns the monodromy of the $\Sigma_g$-bundle over $S^1$ obtained by pulling $X$ back along $\gamma$. The monodromy is well-defined up to conjugating by an element of $Mod(\Sigma_g)$ due to the original identification of a fiber with $\Sigma_g$. A \textit{transitive monodromy} signifies that the map $m_f$ is onto.

In this way, the monodromy $\phi$ of a Lefschetz fibration gives a combinatorial description of the manifold which encodes all of its distinguishable features. For genus-0 and -1 Lefschetz fibrations, all possible monodromies are known. These 4-manifolds are completely classified and have been shown to have total spaces diffeomorphic to $\CP^2$, $\CP^1 \times \CP^1$, elliptic surfaces $E(n)$ and their blow-ups \cite{Kas_80,Moishezon_77}. Lefschetz fibrations of genus $g\geq 2$ are not classified and they will likely remain that way until the complexities of their mapping class groups are more expounded. Fortunately, $Mod(\Sigma_2)$ stands out in this family due to its hyperellipticity. A readers who wishes to read more about hyperellipticity in the setting of mapping class groups are directed to \cite{Margalit_Winarski}, but we will expand on the consequences of this extra physiology in the setting of Lefschetz fibrations. 

\subsection{Hyperelliptic Lefschetz fibrations}\label{Sec:hyperelliptic}
Any simple closed curve on the genus-$g$ surface is \textit{hyperelliptic} if it is isotopic to its mapping under the hyperelliptic involution, an order two automorphism of $\Sigma_g$ with $2g+2$ fixed points. If all vanishing cycles of a genus-$g$ Lefschetz fibration are hyperelliptic, then the Lefschetz fibration is said to be \textit{hyperelliptic}. For $g=2$, all curves are hyperelliptic and therefore all genus-2 Lefschetz fibrations are hyperelliptic. This extra structure allows for the computations of some invariants that are otherwise not as easily computable. The first is the fractional signature formula which has been explored by Matsumoto \cite{Matsumoto_96} and Endo \cite{Endo_signature_2000} (both aided by the results from Meyer \cite{Meyer_73}) and Smith \cite{Smith_Hodge_1999} 
The signature for the genus-2 case, from Theorem 3.3 of \cite{Matsumoto_96} and Theorem 4.8 of \cite{Endo_signature_2000}, is as follows,%
\begin{equation}{\label{eq:fractionsig}}
    \sigma(X) = - \frac{3}{5}n - \frac{1}{5}s 
\end{equation}%
where $n$ and $s$ are the number of separating and nonseparating vanishing cycles in the Lefschetz fibration. It is know a well-known result that %
\begin{equation}{\label{eq:ns_relationship}}
     n+12s \equiv 0 \mod(10),
\end{equation}%
which follows from the Abelianization of $Mod(\Sigma_2)$ being $\Z/{10}\Z$, see Section 5.1.3. of \cite{Farb_Margalit}. Each separating vanishing cycle on $\Sigma_2$ separates the surface into a pair of genus-1 surfaces. On the genus-1 surface with boundary, the length twelve chain relation implies that this separating curve can be written as the composition of twelve right-handed Dehn twists about nonseparating curves. The final relationship between $n$ and $s$ that we will mention is%
\begin{equation}{\label{eq:ns_bounds_BK}}
     n + 7s \geq 20,
\end{equation}%
see Lemma 5 of \cite{Baykur_Korkmaz_2017}. It also holds that knowing $n$ and $s$ is sufficient for calculating the Euler characteristic of a Lefschetz fibration. For genus-2, 
\begin{equation}{\label{eq:eulerchar}}
    \chi(X) = n + s -4
\end{equation}%
which can be found by observing how each vanishing cycle contributes to the handlebody description of the 4-manifold. That is, build the genus-$g$ Lefschetz fibration from $\Sigma_g \times D^2$ and glue in the vanishing cycles (which are 4-dimensional 2-handles glued along $S^1$ with framing $-1$, see Section 8.2 of \cite{Gompf_Stipsicz_1999}). The $\Sigma_g \times D^2$ will have one 0-handle, $2g$ 2-handles, and one 3-handle. There are $k$ 2-handles corresponding to the vanishing cycles. Then, regardless of if there is a section, glue the final $\Sigma_g \times D^2$ which caps off the handlebody (in a way that depends on the existence of a section). This capping $\Sigma_g \times D^2$ will have one 4-handle, $2g$ 3-handles, and one 2-handle. However, we are gluing along boundary $\Sigma_g \times S^1$ and want to avoid double counting as formalized in Mayer-Vietoris. However, $\chi(S^1)$ is zero, so $\chi(X) = 2\chi(\Sigma_g \times D^2) + k = 4-4g$. More details can be found in Section 8.2 of \cite{Gompf_Stipsicz_1999}.

In a Lefschetz fibration, a singular fiber is homotopy equivalent to a regular fiber with a 2-disk attached along its boundary $S^1$ to the corresponding vanishing cycle. The vanishing cycles play a crucial role in distinguishing different Lefschetz fibrations, particularly when computing their fundamental group or first homology. In general, knowing the number of vanishing cycles, even the number of separating and nonseparating vanishing cycles, is not always sufficient for differentiating between them. It can be interesting to study how contradistinct two Lefschetz fibrations of the same type $(n,s_1,s_2, ..., s_k)$ can be, where $n$ is the number of singular fibers and the $s_i$ are the number of separating vanishing cycles that separate $\Sigma_g$ into surfaces of genus-$i$ and genus-$(g-i)$. 

In the case of genus-2 Lefschetz fibrations, for example, it is known that a fibration of type $(4,3)$ has a unique diffeomorphism type, as it corresponds to a surface bundle that has been blown up a finite number of times \cite{Baykur_Korkmaz_2017}. The genus-2 Lefschetz fibrations of type $(6,2), (18,1)$, and $(20,0)$ also have total spaces whose diffeomorphism types are known. But, not all Lefschetz fibrations with genus-2 fibers are diffeomorphic to a ruled (complex) surface like these. In fact, Baykur-Korkmaz \cite{Baykur_Korkmaz_2017} show that if a genus-2 Lefschetz fibration of type $(n,s)$ satisfies $n+7s < 30$, then $X$ must be a rational or ruled, see Equation~\eqref{eq:ns_bounds_BK}. In general, there is no algorithm to decide if a given Lefschetz fibration of type $(n,s_1, s_2,...,s_k)$ exists, is unique, or is indecomposable.

\subsection{Indecomposability of Lefschetz fibrations}\label{Sec:intro_indecomp}
%
A Lefschetz fibration is said to be \textit{trivial} if it has no singular fibers, and \textit{nontrivial} otherwise. It is said to be \textit{decomposable} if it can be described as the fiber sum of nontrivial Lefschetz fibrations. The \textit{fiber sum} is an additive operation of two genus-$g$ Lefschetz fibrations that preserves the fiber and base directions. Suppose $f_1 : X_1 \rightarrow S^2$ and $f_2 : X_2 \rightarrow S^2$ are two nontrivial genus-$g$ Lefschetz fibrations. The fiber sum $X_1 \#_\phi X_2$ is defined by removing a regular neighborhood $C_i = \Sigma_g \times D^2$ from $X_i, i=1,2$ and gluing $X_1 - C_1$ to $X_2 - C_2$ by $\phi$, an orientation reversing diffeomorphism. A Lefschetz fibration is \textit{indecomposable} if whenever it is written as a fiber sum, one of the $X_i$ must always be a trivial Lefschetz fibration. 

\begin{figure}[!ht]\label{Fig:Lantern}
    \centering
    \begin{tikzpicture}[scale=2]%
    \draw[line width=.5mm] (0.25,0) to[in=170, out=10] (1.25,0);
    \draw[line width=.5mm] (0.25,1.5) to[in=190, out=-10] (1.25,1.5);
    \draw[line width=.5mm] (0,.25) to[in=-80, out=80] (0,1.25);
    \draw[line width=.5mm] (1.5,0.25) to[in=-100, out=100] (1.5,1.25);
    \draw[line width=.5mm] (0.25,0) to[in=45,out=45] (0,0.25);
    \draw[line width=.5mm] (0.25,0) to[in=-135,out=-135] (0,0.25);
    \draw[line width=.35mm, blue]        (0.45,0.05) to[in=45,out=45]    (0.05,0.45);
    \draw[line width=.35mm, blue,dashed] (0.45,0.05) to[in=-135,out=-135] (0.05,0.45);
    \draw[line width=.5mm] (1.25,0) to[in=-45,out=-45] (1.5,0.25);
    \draw[line width=.5mm] (1.25,0) to[in=135,out=135] (1.5,0.25);
    \draw[line width=.35mm, blue, dashed] (1.05,0.05) to[in=-45,out=-45] (1.45,0.45);
    \draw[line width=.35mm, blue]         (1.05,0.05) to[in=135,out=135]  (1.45,0.45);    
    \draw[line width=.5mm] (0,1.25) to[in=135,out=135] (.25,1.5);
    \draw[line width=.5mm] (0,1.25) to[in=-45,out=-45] (.25,1.5);
    \draw[line width=.35mm,blue,dashed] (0.05,1.05) to[in=135,out=135] (.45,1.45);
    \draw[line width=.35mm,blue]        (0.05,1.05) to[in=-45,out=-45] (.45,1.45);  
    \draw[line width=.5mm] (1.5,1.25) to[in=45,out=45] (1.25,1.5);
    \draw[line width=.5mm] (1.5,1.25) to[in=-135,out=-135] (1.25,1.5);
    \draw[line width=.35mm, blue,dashed] (1.45,1.05) to[in=45,out=45]    (1.05,1.45);
    \draw[line width=.35mm, blue]        (1.45,1.05) to[in=-135,out=-135] (1.05,1.45);    
    \draw[line width=.5mm] (02.25,0) to[in=170, out=10] (3.25,0);
    \draw[line width=.5mm] (02.25,1.5) to[in=190, out=-10] (3.25,1.5);
    \draw[line width=.5mm] (2,.25) to[in=-80, out=80] (02,1.25);
    \draw[line width=.5mm] (3.5,0.25) to[in=-100, out=100] (3.5,1.25);
    %
    \draw[line width=.5mm] (2.25,0) to[in=45,out=45] (2,0.25);
    \draw[line width=.5mm] (2.25,0) to[in=-135,out=-135] (2,0.25);
    %
    \draw[line width=.5mm] (3.25,0) to[in=-45,out=-45] (3.5,0.25);
    \draw[line width=.5mm] (3.25,0) to[in=135,out=135] (3.5,0.25);  
    %
    \draw[line width=.5mm] (2,1.25) to[in=135,out=135] (2.25,1.5);
    \draw[line width=.5mm] (2,1.25) to[in=-45,out=-45] (2.25,1.5); 
    %
    \draw[line width=.5mm] (3.5,1.25) to[in=45,out=45] (3.25,1.5);
    \draw[line width=.5mm] (3.5,1.25) to[in=-135,out=-135] (3.25,1.5);
    \draw[line width=.35mm,orange]        (02.75,0.05) to[in=-30,out=30] (2.75,1.45);
    \draw[line width=.35mm,orange,dashed] (02.75,0.05) to[in=-150,out=150] (2.75,1.45);
    \draw[line width=.35mm,cyan]        (02.05,0.75) to[in=120,out=60]   (3.45,0.75);
    \draw[line width=.35mm,cyan,dashed] (02.05,0.75) to[in=-120,out=-60] (3.45,0.75);
    \draw[line width=.35mm,violet]        (2.75,1.45)  to[out=180,in=90]  (02.05,0.75);
    \draw[line width=.35mm,violet,dashed] (02.05,0.75) to[out=-90,in=180] (02.75,0.05);
    \draw[line width=.35mm,violet]        (02.75,0.05) to[out=0,in=-90]   (3.45,0.75);
    \draw[line width=.35mm,violet,dashed] (3.45,0.75)  to[out=90,in=0]    (2.75,1.45);
    \node (c01) at (.56,1.2) {$\partial_1$};
    \node (c03) at (.56,0.35) {$\partial_2$};
    \node (c02) at (.95,1.2) {$\partial_3$};
    \node (c04) at (.95,0.35) {$\partial_4$};

    \node (c2) at (2.34,1.2) {$c_2$};
    \node (c3) at (2.8,1.0) {$c_3$};
    \node (c4) at (2.99,0.6) {$c_4$};
    \end{tikzpicture}
    \caption{Lantern relation, $\tau_{\partial_1} \tau_{\partial_2}\tau_{\partial_3} \tau_{\partial_4} = \tau_A \tau_B \tau_C$. Note that the separating curve depends on the embedding of the four-punctured sphere into the fiber surface.}
    \label{tikz_curvesab}
\end{figure}
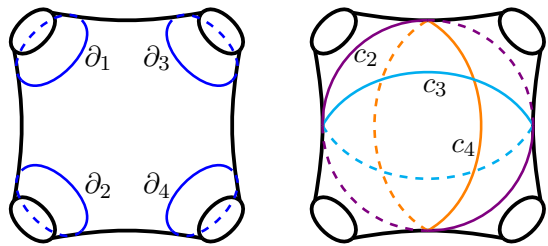%

Oftentimes, new Lefschetz fibrations are uncovered by taking commonly known chain relations as monodromies and altering them using Hurwitz moves (elementary transformations), braid relations, chain relation substitutions, and generalized lantern substitutions. Generalized lantern relations are generalizations of Figure~\ref{Fig:Lantern} with more boundary components. A reader is directed to \cite{Margalit_Winarski} for more information about these relations. Although the vanishing cycles on this possibly ``new" Lefschetz fibration can be somewhat complicated as curves, these techniques are still useful for finding particularly illusive Lefschetz fibrations, such as the type $(4,3)$ \cite{Baykur_Korkmaz_2017} and the type $(10,10)$ \cite{Nakamura_2018} genus-2 Lefschetz fibrations. It is important to note that a genus-2 lantern relation substitution will replace four nonseparating vanishing cycles with two different nonseparating and one separating vanishing cycles. For a genus-2 of type $(n,s)$, performing a lantern will alter it to a type $(n-2, s+1)$. For example, the genus-2 type $(10,10)$ Lefschetz fibration found by Nakamura \cite{Nakamura_2018} is found after fiber summing three copies of the $(4,3)$ and performing one lantern relation on that resulting type $(12,9)$. 

\section{Results on nontrivial Lefschetz fibrations}

For any pair of nontrivial Lefschetz fibrations with the same fiber genus, their monodromies capture their most distinguishing topological features. Many rich results about Lefschetz fibrations are derived using the least amount data extracted from the monodromy, since the problem of finding all positive factorizations of the identity of a mapping class group is complex. Equally as difficult would be computing all invariants of a nontrivial genus-$g$ Lefschetz fibration without explicitly identifying its vanishing cycles. 

In the hyperelliptic case, the task is significantly simplified, but for non-hyperelliptic Lefschetz fibrations, it is often the case that more sophisticated tools are required. For instance, Ozbagci provides an algorithm for computing the signature based on Wall’s non-additivity formula, which involves ordering the singular fibers and then computing the contribution to the signature of each singular fiber based on the previous singular fibers \cite{Ozbagci_signatures}. This is a difficult process for all but the simplest Lefschetz fibrations. Additionally, Smith offers an explicit formula for the signature that depends only on the number of vanishing cycles and the Euler number of the associated Hodge line bundle \cite{Smith_Hodge_1999}. In comparison, the proof of Theorem~\ref{ThmTransitivity} is noteworthy because it does not require exact knowledge of the vanishing cycles. 

Recall that Theorem~\ref{ThmTransitivity} says that any genus-$g$ Lefschetz fibration with transitive monodromy is simply connected. The proof does not rely on the underlying structure of the symplectic 4-manifold, but rather the geometric consequences of this extra structure on the monodromy with respect to the mapping class group of the fiber.%
\begin{proof}[Proof of Theorem~\ref{ThmTransitivity}]

    Recall the construction of a Lefschetz fibration $f: X \rightarrow S^2$ from Section~\ref{Sec:hyperelliptic}. This handlebody description of $X$ implies that the fundamental group is a quotient of the fundamental group of $\Sigma_g$ by relationship corresponding to the vanishing cycles. In the case that $f$ has no section, there is possibly another relation coming from the gluing the capping $\Sigma_g \times D^2$ to the base $\Sigma_g \times D^2$. We see that if we can show that every embedded simple closed curve on a regular fiber of the Lefschetz fibration is nullhomotopic, then it follows that $X$ is simply connected. 

    Recall the global monodromy of the Lefschetz fibration
    \begin{equation*}
        m_f: \pi_1 \left (S^2-\{q_1, \ldots, q_k \}, x_0 \right ) \rightarrow Mod(\Sigma_g).
    \end{equation*}
    discussed in Section~\ref{Sec:hyperelliptic}. Suppose $\eta$ is a nonseparating vanishing cycle on $F_p$, a regular fiber. Then for any nonseparating simple closed curve $\gamma$ on $F_p$, by transitivity, there is a (possibly immersed) closed curve $\delta$ in $S^2-\{q_1, \ldots, q_k \}$ such that $m(\delta)$ sends $\gamma$ to a curve isotopic to $\eta$. 
    
    Think of $\delta$ as the map $\delta:[0,1]\rightarrow X$ and pull-back $X$ to $[0,1]$ to get a trivial $\Sigma_g$-bundle over $[0,1]$. Now consider $\eta \in \Sigma_g\times\{1\}$ and let $A=\eta\times[0,1]$. The image of $A$ under the map from $\delta(X)$ to $X$ is an immersed annulus with one boundary component $\eta$ and the other isotopic to $\gamma$. Now, adding the disk, the vanishing cycles bounds to this annulus will give an immersed disk with boundary $\gamma$. So $\gamma$ is nullhomotopic. The same argument is made for every nonseparating curve on $\Sigma_g$, therefore the 4-manifold is simply connected. %
\end{proof}%

Recall the conjecture that for genus-2 Lefschetz fibrations, the fundamental group must have at most two generators. It is possible to say more about this question in the case of the Abelianization of the fundamental group. To do so, we need to understand the relationship between the algebraic intersection number of curves on the surface $\Sigma_g$ and the first homology, $H_1(\Sigma_g;\Z)$, viewed as a symplectic vector space. 

Let $\{ [a_1], [b_1], ..., [a_g], [b_g]\}$ be an ordered basis for $H_1 \left(\Sigma_g; \Z \right)$ such that the algebraic intersection between the curves satisfies $\hat{i}{([a_i], [b_j])}$ is $-1$ if and only if $i=j$ and all other intersection numbers are zero. Note that $\hat{i}{([b_j], [a_i])}$ is $1$. Here \[ \hat{i}:H_1(\Sigma_g;\Z)\times H_1(\Sigma_g;\Z)\to \Z \] is the algebraic intersection number between homology classes. Thus $H_1(\Sigma_g;\Z)$ is a symplectic vector space. Since any diffeomorphism of $\Sigma_g$ must preserve the intersection number between curves, it is an automorphism of the symplectic vector space $H_1(\Sigma_g;\Z)$. Thus any diffeomorphism gives an element in the symplectic group $Sp(2g,\Z)$. Therefore, we have a symplectic representation
\[ \Psi: Mod \left(\Sigma_g \right) \rightarrow Sp(2g; \Z),\] see \cite{Farb_Margalit}. The following fact will be useful. 
\begin{fact}\label{sympl_rep_thm}
    For isotopy classes of oriented simple closed curves $a$ and $b$ in $\Sigma_g$, and for all $k \geq 0$, $\Psi\left( \tau_b^k \right) \left( [a] \right) = [a] + k \cdot \hat{i}(a,b)[b]$.
\end{fact}
We are now ready to prove Theorem~\ref{ThmBettinum}. Recall it says that for $X$ a nontrivial genus-$g$ Lefschetz fibration, $b_1(X) \leq 2g-2$.

\begin{proof}[Proof of Theorem~\ref{ThmBettinum}]
    Recall the handlebody construction of a Lefschetz fibration $f: X \rightarrow S^2$ as seen in Section~\ref{Sec:hyperelliptic}. Only the closed surface is producing any nontrivial first homology in $\Sigma_g \times D^2$. The homology $H_1(X)$ is a finitely presented group generated by the homological basis elements of $\Sigma_g$, call them $\{ [a_1], [b_1], [a_2], [b_2], \ldots, [a_g], [b_g] \}$. The 2-handles attached along the vanishing cycles are some integral linear combination of these homology classes and will be relations in the presentation. If there is no section, there may be another relation. So, when computing the first homology of a Lefschetz fibration, it is enough to know what the vanishing cycles are (and possibly the section).

    Due to Smith \cite{Smith_Hodge_1999}, it is known that every nontrivial genus-$g$ Lefschetz fibration has at least one nonseparating vanishing cycle. (They show this in their proof of Theorem 6.2 which claims that there are no Lefschetz fibrations with monodromy group a subgroup of the Torelli group.) Label the single nonseparating vanishing cycle $\eta$. The first Betti number of $X$ will now be at most $2g-1$ after identifying $\eta$ with a basis element of the first homology. 

    Assume for contradiction that all remaining vanishing cycles, $\eta_1, \eta_2, ..., \eta_{s-1}$, are all simple closed curves in the same homology class as $[\eta]$ (Note that for $s=1$, the following argument still holds). This does not change the fact that $b_1(X) = 2g-1$. Then, Remark~\ref{sympl_rep_thm} implies that for every class of curves $[a]$ on the surface,
    \[ [a] = \Psi \left( Id \right) \left( [a] \right) = \Psi \left ( \tau_{\eta_s} ... \tau_{\eta_2} \tau_{\eta_1} \right ) \left ( [a] \right) = [a] + s \cdot \hat{i}(a,\eta)[\eta] \]
    However, there are classes of curves on the surface such that $s \neq 0$ and so it must be true that $\hat{i}(a,\eta) \neq 0$. Thus this equality is not always true like we require it to be. Therefore, there must be at least one other vanishing cycle realizing a homology class distinct from $[\eta]$ which aides in taking a homological class of curves back to itself. 
\end{proof}

\begin{remark}
    Stipsicz \cite{Stipsicz_99} improved Smith's result \cite{Smith_Hodge_1999} to at least $\frac{4g+2}{5}$ vanishing cycles but does not comment on the possible homological behavior of the curves. Of these $\frac{4g+2}{5}$ vanishing cycles, Theorem~\ref{ThmBettinum} states that at least two of them are nonhomologous and nonseparating.
\end{remark}

This theorem has a similar statement for Lefschetz fibrations over higher genera surfaces, $f:\Sigma_g \rightarrow \Sigma_h$.  

\begin{theorem}\label{Thm:GeneralizedBettinum}
    For a Lefschetz fibration $X$ over a higher genera base space, $f : X \rightarrow \Sigma_h$, the inequality on the first Betti number of $X$ is as follows. 
    \[ 0 \leq b_1(X) \leq 2g + 2h -2\]
\end{theorem}

\begin{proof}
    The above proof immediately generalizes to prove this, but we give a different argument here. 
    
    The Lefschetz fibration structure on $X$ with base space $\Sigma_h$ can be thought of as the fiber sum of $M_1$ and $M_2$, where $M_1$ is the Lefschetz fibration $\widehat{f}: M_1 \rightarrow S^2$ with the same vanishing cycles as $f:X \rightarrow \Sigma_h$ and $M_2$ is a $\Sigma_h$-bundle over $\Sigma_g$. Then, one can see that the rank of first homology from the fiber direction is unchanged, and the base space has increased homology coming from the addition of the new handles on the surface. 
\end{proof}

Recall that for genus-$2$ Lefschetz fibrations over the sphere, this bound is sharp. There exist Lefschetz fibrations that realize $b_1(X) =0,1,$ and $2$, see \cite{Baykur_Korkmaz_2017, Amoros_ABKP_99, Korkmaz2009, Ozbagci_Stipsicz_2000}. The current conjecture for fundamental groups of genus-2 Lefschetz fibrations is that they are at most rank two and Abelian. For nontrivial genus-2 Lefschetz fibrations with $b_1(X)=2$, if there are two homology classes of vanishing cycles $[\eta_1]$ and $[\eta_2]$ which are a part of an integral basis for homology, then it must be true that the first homology is a quotient of $\Z^2$. In this case, the equivalence relation brought on by $[\eta_1]$ and $[\eta_2]$ will force at most two generators of the Abelianization of the fundamental group. This is a restatement of Problem~\ref{Prob:hombasis}.

\section{Applications to genus-2 Lefschetz fibrations}\label{sec:genus_2}

Hyperelliptic Lefschetz fibrations are a well-structured subset of symplectic 4-manifolds. Specifying the quantity of nonseparating and separating curves that make up the vanishing cycles is enough to compute most invariants such as the Euler characteristic and signature, or inequalities between them, see \cite{Altunoz_2020, Baykur_Korkmaz_2017,  Smith_Hodge_1999}. All genus-2 Lefschetz fibrations have vanishing cycles which are hyperelliptic, and the results in this section rely on this fact.

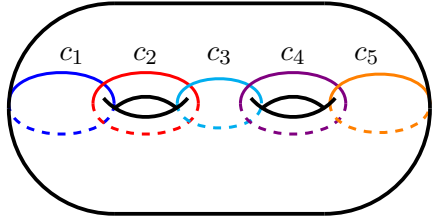
\begin{figure}[h!]
    \centering
   \begin{tikzpicture}[scale=1.4]
    \draw[line width=.4mm, blue] (0,1.05) to[in=90, out=90] (1,1.05);
    \draw[line width=.4mm, red] (.8,1.05) to[in=90, out=90] (1.8,1.05);
    \draw[line width=.4mm, cyan] (1.6,1.05) to[in=90, out=90] (2.4,1.05);
    \draw[line width=.4mm, violet] (2.2,1.05) to[in=90, out=90] (3.2,1.05);
    \draw[line width=.4mm, orange] (3.05,1.05) to[in=90, out=90] (4,1.05);
    \draw[line width=.4mm, blue, dashed] (0,1.05) to[in=-90, out=-90] (1,1.05);
    \draw[line width=.4mm, red, dashed] (.8,1.05) to[in=-90, out=-90] (1.8,1.05);
    \draw[line width=.4mm, cyan, dashed] (1.6,1.05) to[in=-90, out=-90] (2.4,1.05);
    \draw[line width=.4mm, violet,dashed] (2.2,1.05) to[in=-90, out=-90] (3.2,1.05);
    \draw[line width=.4mm, orange, dashed] (3.05,1.05) to[in=-90, out=-90] (4,1.05);

    \draw[line width=.5mm] (0,1) to[in=180, out=-90] (1,0);
    \draw[line width=.5mm] (0,1) to[in=180, out=90] (1,2);
    \draw[line width=.5mm] (1,0) to[in=180, out=0] (3,0);
    \draw[line width=.5mm] (1,2) to[in=180, out=0] (3,2);
    \draw[line width=.5mm] (4,1) to[in=0, out=90] (3,2);
    \draw[line width=.5mm] (4,1) to[in=0, out=-90] (3,0);
    \draw[line width=.5mm] (1,1) to[in=140, out=40] (1.6,1);
    \draw[line width=.5mm] (.9,1.1) to[in=-130, out=-50] (1.7,1.1);
    \draw[line width=.5mm] (2.4,1) to[in=140, out=40] (3,1);
    \draw[line width=.5mm] (2.3,1.1) to[in=-130, out=-50] (3.1,1.1);
    \node (c1) at (.6,1.5) {$c_1$};
    \node (c2) at (1.3,1.5) {$c_2$};
    \node (c3) at (2,1.5) {$c_3$};
    \node (c4) at (2.7,1.5) {$c_4$};
    \node (c5) at (3.4,1.5) {$c_5$};
       \end{tikzpicture}
    \label{fig:chainrltn}
    \caption{The standard curves on the genus-2 surface.}
\end{figure}

Recall that a \textit{K\"{a}hler manifold} is a 4-manifold $X$ with a symplectic structure $\omega$ and an integrable complex structure $J$ such that $g(u,v) = \omega(u,Jv)$ for $u,v \in TX$ and $J: TX \rightarrow TX$ and $J^2=-1$. Note that a symplectic 4-manifold comes equipped with an almost-complex structure. Lemma~\ref{ThmHolomorphic} combines a result of Smith and Chakiris with that of Siebert and Tian to say that in the case of only nonseparating vanishing cycles, a nontrivial genus-2 Lefschetz fibration is holomorphic (or K\"{a}hler) if and only if its monodromy is transitive. That is, Chakiris, and then later Smith, shows that nontrivial genus-2 Lefschetz fibrations with only nonseparating vanishing cycles must have monodromy a composition of the following three relations: %
\begin{equation}
    \begin{split}
        \alpha &= (\tau_1 \tau_2 \tau_3 \tau_4 \tau_5)^6 \\
        \beta &= (\tau_1 \tau_2 \tau_3 \tau_4)^{10} \\
        \gamma &= (\tau_1 \tau_2 \tau_3 \tau_4 \tau_5 \tau_5 \tau_4 \tau_3 \tau_2 \tau_1 )^2
    \end{split}
\end{equation}%
where $\tau_i$ is the right-handed Dehn twist about the curve $c_i$ pictured in Figure~\ref{fig:chainrltn}. 

In the setting of having a holomorphic structure and only nonseparating curves in the monodromy (which has to be made up of $\alpha, \beta, $and $\gamma$), it is true that the monodromy will be transitive. The curves in Figure~\ref{fig:chainrltn} are also a part of the Humphries generators for the mapping class group $Mod(\Sigma_2)$. Theorem~\ref{ThmTransitivity} states that these holomorphic nontrivial genus-2 Lefschetz fibrations with only nonseparating vanishing cycles are always simply connected. 

Now we prove Theorem~\ref{ThmBound_ns}. Baykur and Korkmaz \cite{Baykur_Korkmaz_2017}  combine previously known inequalities relating $n$ and $s$, the number of nonseparating and separating vanishing cycles in a genus-2 Lefschetz fibration respectively, to recover the following inequality. %
\begin{equation}
    2n - s \geq 3
\end{equation} %
This comes from Equation 4 in Lemma 5 of \cite{Baykur_Korkmaz_2017} but can be improved by substituting $b_1 \leq 3$ to the sharper bound on the first Betti number from Theorem~\ref{ThmBettinum}. 

\begin{equation}\label{eq_sharpline}
\begin{split}
        b_1(X) & \geq 4 - \frac{2}{5}(2n-s) \\
        2    & \geq 4 - \frac{2}{5}(2n-s) \\
        2n-s & \geq 5
\end{split}
\end{equation}

The first line of Equation~\eqref{eq_sharpline} comes from combining Equation~\eqref{eq:eulerchar} and a relation between $n$, $s$, and $b^-(X)$ motivated by the subspace in the second homology of the manifold generated by embedded tori found in singular fibers, see Equation 3 in Lemma 5 of \cite{Baykur_Korkmaz_2017}. This improvement on the first Betti number geometrically moves the bound uncovered by Baykur and Korkmaz to intersect actual $(n,s)$ coordinates that satisfy $n + 12s \equiv 0 (\text{mod }10)$, Equation~\eqref{eq:ns_relationship}. 

Baykur and Korkmaz's line is then moved to the line that contains the $(2k, 4k-5)$ family of Lefschetz fibrations, of which only the monodromy of the type $(4,3)$ has been found. This line shift is visualized in Figure~\ref{fig:graph}. The Lefschetz fibrations on the line $2n - s = 5$ have $b_1 \geq 2$, and therefore $b_1 =2$ due to Theorem~\ref{ThmBettinum}. More can be found about the smallest genus-2 Lefschetz fibration of type $(4,3)$ here \cite{Baykur_Korkmaz_2017}. It has fundamental group $\Z^2$ and it found by performing a sequence of braid relations, Hurwitz moves, and chain relation substitutions. The smallest genus-2 Lefschetz fibration was also found indepedently by Xiao in \cite{Xiao_85}.

\begin{proof}[\bf Theorem~\ref{ThmBound_ns}]{\em
Suppose $f:X\rightarrow S^2$ is a genus-2 Lefschetz fibration with $n$ nonseparating and $s$ separating vanishing cycles. Then $2n-s \leq 5$ and the family of Lefschetz fibrations of type $(n,s)=(2k, 4k-5), k\geq 2$ is indecomposable, has $b_1 = 2$ for all $k$, and has $b_2 = n+s-2$. \phantom\qedhere }
\end{proof}

\begin{proof}
    Theorem~\ref{ThmBettinum} states that $b_1(X) < 3$ for nontrivial genus-2 Lefschetz fibrations. 

    Suppose $b_1(X) = 1$. Then, using Equation~\eqref{eq:fractionsig} and Equation~\eqref{eq:eulerchar}, we see the following.
    \begin{equation}
        \begin{split}
            b_2^+(X) &= \frac{1}{2}\Big(\chi(X) + \sigma(X)\Big) + b_1(X) -1 \\
            &= \frac{1}{2}\left ( n + s - 4 - \frac{3}{5}n - \frac{1}{5}s \right) + 1 -1\\
            &= \frac{1}{5}n + \frac{2}{5}s -2
        \end{split}
    \end{equation}

    Recall, $\sigma(X) = b_2^+(X) - b_2^-(X)$ and therefore,
    \begin{equation}
        \begin{split}
            b_2^-(X) &= b_2^+ - (- \frac{3}{5}n - \frac{1}{5}s )\\
            &= \frac{1}{5}n + \frac{2}{5}s - 2 + \frac{3}{5}n + \frac{1}{5}s \\
            &= \frac{4}{5}n + \frac{3}{5}s - 2.
        \end{split}
    \end{equation}
    Since we consider only those Lefschetz fibrations sitting on the line $2n-s=5$,
    \begin{equation}
        \begin{split}
            b_2^-(X) &= \frac{2}{5}(s+5) + \frac{3}{5}s - 2 \\
            &= s.
        \end{split}
    \end{equation}
    However, it is known \cite{Baykur_Korkmaz_2017} that each separating vanishing cycle contributes to $H_2(X)$. In fact, $b_2^-(X) \geq s + 1$, see Equation 2 in Lemma 5 of \cite{Baykur_Korkmaz_2017}. So this is a contradiction. A similar contradiction ($b_2^- = s-1)$ is found when $b_1(X) =0$, so therefore $b_1(X) \geq 2$. Thus, $b_1(X) =2$ for $X$ in this family.

    All $(n,s)$ coordinates on this line satisfy the three conditions listed in Equation~\eqref{eq:fractionsig}, Equation~\eqref{eq:ns_relationship}, and Equation~\eqref{eq:ns_bounds_BK}. These Lefschetz fibrations are indecomposabe (if they exist) because there is no way to fiber sum a Lefschetz fibration $X_1$ of type $(n_1, s_2)$ to $X_2$ of type $(n_2, s_2)$ such that $(n_1 + n_2, s_1 + s_2)$ is a vector with a steeper slope than its summand parts. 

    Finally, we remark that for Lefschetz fibrations in this family with type $(n,s)$ as in the statement of the Theorem, $b_2^+(X) = n-3$, $b_2^-(X)=2n-4$, and therefore $b_2(X)=3n-7$, $n + s-2$. We begin with the same equation from above to show $b_2^+$ and the remaining terms follow. 

    \begin{equation}
        \begin{split}
            b_2^+(X) &= \frac{1}{2}\Big(\chi(X) + \sigma(X)\Big) + b_1(X) -1 \\
            &= \frac{1}{2}\left ( n + s - 4 - \frac{3}{5}n - \frac{1}{5}s \right) + 2 -1\\
            &= \frac{1}{5}n + \frac{2}{5}s -1\\
            &= \frac{1}{5}n + \frac{2}{5}\left( 2n-5\right) -1 \\
            &= n-3
        \end{split}
    \end{equation}
\end{proof}

In the following plot, possible genus-2 Lefschetz fibrations are marked, with filled in marks, {\small{$\newmoon$}}, for genus-2 Lefschetz fibrations whose monodromies have been found. Note that there are no nontrivial genus-2 Lefschetz fibrations with monodromies of length less than 7, and none of length 9, 10, 11, or 12. This can be seen using Equation~\eqref{eq:ns_relationship} and Equation~\eqref{eq_sharpline}. Note that the monodromy of length 13 with type $(6,7)$ has not yet been found. 

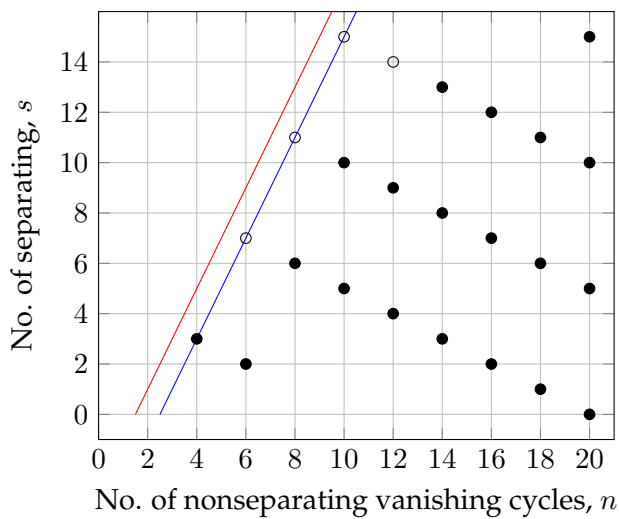
\begin{figure}[!ht]
    \centering
  \begin{tikzpicture}
    \begin{axis}[%
        scatter/classes={%
        b={mark=o,draw=black}, a={draw=black}},
        xmin=0, xmax=21,
        ymin=-1, ymax=16,
        xtick={0,2,4,6,8,10,12,14,16,18,20},
        ytick={0,2,4,6,8,10,12,14},
        xlabel = {No. of nonseparating vanishing cycles, $n$},
        ylabel = {No. of separating, $s$},
        xmajorgrids=true,
        ymajorgrids=true
    ]
    \addplot [
        domain=1.5:10,
        samples=100, 
        color=red,
    ]
    {2*x-3};
    \addplot [
        domain=2.5:11, 
        samples=100, 
        color=blue,
    ]
    {2*x-5};
    \addplot[only marks, mark=*] coordinates { (4,3) (6,2) (8,6) (10,5) (10,10)
  (12,4) (12,9) (14,3) (14,8) (14,13) (16,2) (16,7) (16,12)
  (18,1) (18,6) (18,11) (20,0) (20,5) (20,10) (20,15) };

    \addplot[only marks, mark=o] coordinates { (6,7) (8,11) (10,15) (12,14)
    };
    \addlegendimage{mark=ball,draw=white}
    \addlegendimage{}
    \end{axis}
  \end{tikzpicture}
\caption{Genus-2 Lefschetz fibrations of type $(n,s)$. Red line is $2n-s=3$, Blue line is $2n-s=5$, the mark {\small{$\newmoon$}} indicates a type $(n,s)$ Lefschetz fibration whose monodromy has been explicitly found, the mark {\small{$\fullmoon$}} indicates a possible $(n,s)$ whose monodromy has not been found.}
    \label{fig:graph}
\end{figure}

Here we discuss the indecomposable Lefschetz fibrations in Figure~\ref{fig:graph} which have been found.  
\begin{itemize}
    \item The type $(6,2)$ Lefschetz fibration is Matsumoto's genus-2 relation found by splitting singular fibers using Mathematica \cite{Matsumoto_2004}. 
    \item The type $(4,3)$ is found by Baykur and Korkmaz, who show it is the smallest monodromy for any genus-2 Lefschetz fibration \cite{Baykur_Korkmaz_2017}.
    \item The all nonseparating $(20,0)$ Lefschetz fibration has monodromy the square of the hyperelliptic involution $(\tau_{5}\tau_{4}\tau_{3}\tau_{2}\tau_{1}\tau_{1}\tau_{2}\tau_{3}\tau_{4}\tau_{5})^2 =1$ which can be found in Chapter 5 of \cite{Farb_Margalit}. See Figure~\ref{fig:chainrltn} for the curves $c_i, i=1,\ldots,5$.
    \item From the $(20,0)$, one or two lantern relations (which replace four nonseparating vanishing cycles with two nonseparating and one separating) can be performed to recover, respectively, the type $(18,1)$ and type $(16,2)$ Lefschetz fibration. 
    \item The type $(14,3)$ is found by Endo and Gurtas in \cite{Endo_Gurtas} in Example 5.3. They perform four lantern relations on the $(20,0)$ hyperelliptic involution to recover an explicit monodromy of $(12,4)$, but by omitting the last relation they have shown that $(14,3)$ exists. We remark that $(12,4)$ can also be recovered as the fiber sum of two copies of Matsumoto's $(6,2)$ Lefschetz fibration.  
\end{itemize}
The remaining of the filled in marks can be written as fiber sums of copies of $(6,2)$ and/or $(4,3)$ with the indecomposable Lefschetz fibrations listed just above. A coordinate of particular interest is that of type $(10,10)$ uncovered by Nakamura who fiber sums three copies of the $(4,3)$ with itself to get a type $(12,9)$. They then perform one lantern substitution after a sequence of Hurwitz moves which then recovers the $(10,10)$ Lefschetz fibration. Nakamura shows the curves which are replaced in the lantern relation and argues that this Lefschetz fibration is on a simply connected, minimal symplectic 4-manifold \cite{Nakamura_2018}. We remark that the Lefschetz fibrations of type $(14,13)$ is discussed by Akhmedov and Monden in \cite{Akhmedov_Monden_21}. 

It is often interesting to consider Lefschetz fibrations with small $b_2^+$. For example, Stipsicz \cite{Stipsicz_2002} considers genus-$g$ Lefschetz fibrations with $b_2^+(X) =1$ as they are the likely the best candidates for finding the genus-$g$ Lefschetz fibration with the fewest number of singular fibers. 

\begin{lemma}\label{Lem:b2+}
    There are only nine types of nontrivial genus-2 Lefschetz fibrations with $b_2^+ = 1$ and of those nine, seven must have $b_1(X) = 0$. 
\end{lemma}

\begin{proof}
    We list the nine Lefschetz fibrations. 
    \begin{itemize}
        \item First Betti number $b_1(X) = 2$: $(4,3), (6,2)$
        \item First Betti number $b_1(X)=0$: $(8,6), (10,5) ,(12,4) ,(14,3),(16,2),(18,1),(20,0)$
    \end{itemize}

    In the case of $b_1(X)=2$, the Lefschetz fibrations live on the line $\frac{1}{2}n + s = 5$. For the seven types with $b_1(X)=0$, the manifolds live on the line $\frac{1}{2}n + s = 10$. All other nontrivial genus-2 Lefschetz fibrations live on a line $\frac{1}{2}n + s = 5k$ for $k > 2$. With the exception of $(6,7)$, a genus-2 Lefschetz fibration with $b_1(X)=2$ which has not yet been shown to exist and would have to have $b_2^+= 3$ if it does.

    Recall from the proof of Theorem~\ref{ThmBettinum} the relationship between $b_2^+(X)$, $n$ and $s$, and $b_1(X)$. 
    \begin{equation}
    \begin{split}
        b_2^+(X) &= \frac{1}{5}n + \frac{2}{5}s + b_1(X) - 3\\
        1 &= \frac{1}{5}n + \frac{2}{5}s + b_1(X) -3
    \end{split}
    \end{equation}

    Now, replace $n$ with $10k-2s$ to indicate that we are only considering Lefschetz fibrations with types on the line $\frac{1}{2}n + s = 5k, k=1,2$. 
    
    \begin{equation}
    \begin{split}
        4 &= \frac{1}{5}(10k-2s) + \frac{2}{5}s + b_1(X)\\
        4 &= 2k + b_1(X)
    \end{split}
    \end{equation}
    
    Therefore, 

    \begin{equation}
    \begin{split}
        b_1(X) = 4-2k,
    \end{split}
    \end{equation}

    but we know from Theorem~\ref{ThmBettinum} that either $b_1(X) = 0$ or $b_1(X) = 2$. So, $k$ must equal $1$ or $2$. 
    
    Note that these are the only type $(n,s)$ where $b_2^+(X)=1$, because otherwise they live on a line where $k > 2$.
\end{proof}

\begin{remark}
    The Lefschetz fibration of type $(6,7)$ does not follow the proof of Lemma~\ref{Lem:b2+} as it contradicts the fact that $b_2^- \geq s + 1$. When $b_1(X) = 0$ or $1$, $b_2^+(X) = 6$ and $7$ respectively. Neither are larger than $7$. It can be checked using the signature formula for genus-2 Lefschetz fibration, seen in Equation~\eqref{eq:fractionsig}, that this inequality holds for all nine types listed in the above Lemma. Sato \cite{Sato_2010} observes $b_1^+$ in this family but does not specify the $(n,s)$ combinations within it.
\end{remark}

\bibliographystyle{plain} 
\bibliography{cited}
\end{document}